\documentclass[a4paper,11pt]{amsart}

\usepackage{amsfonts,amssymb}
\usepackage{graphicx,tikz}
\usepackage{tikz-cd} 
\usepackage{float}
\usepackage{amsmath, amsthm, amsfonts}
\usepackage{geometry}
\usepackage{hyperref}
\usepackage{enumitem}  
\usepackage[capitalise]{cleveref}
\usepackage{comment}
\usepackage{enumitem}
\usepackage{amsrefs}
\usepackage{nicefrac}
\usepackage[official]{eurosym}
\usepackage{graphicx}
\usepackage{nomencl}
\usepackage{amsfonts}
\usepackage{hyperref}
\usepackage{amssymb}
\usepackage{fancyhdr}
\usepackage{amscd}
\usepackage[english]{babel}

\theoremstyle{plain}
\theoremstyle{definition}

\newtheorem{theoremletter}{Theorem}

\newtheorem{propositionletter}[theoremletter]{Proposition}

\newtheorem{thm}{Theorem}[section]

\newtheorem{cor}[thm]{Corollary}
\newtheorem{lem}[thm]{Lemma}
\theoremstyle{definition}

\makeatletter
\def\cleardoublepage{\clearpage\if@twoside \ifodd\c@page\else
	\hbox{}
	\thispagestyle{empty}
	\newpage
	\if@twocolumn\hbox{}\newpage\fi\fi\fi}
\makeatother
\DeclareMathOperator{\Aut}{Aut}
\DeclareMathOperator{\St}{St}

\DeclareMathOperator{\h}{hdim}

\def\G{\mathcal{G}}

\keywords{Groups acting on trees, branch groups, Hausdorff dimension}

\begin{document}
	
	\title[Hausdorff dimension of the second Grigorchuk group]{Hausdorff dimension of the second Grigorchuk group}

\author[M. Noce]{Marialaura Noce}
\address{Marialaura Noce: Department of Mathematical Sciences, University of Bath, Claverton Down, Bath BA2 7AY, United Kingdom}
\email{mnoce@unisa.it}

 \author[A. Thillaisundaram]{Anitha Thillaisundaram}
 
 \address{Anitha Thillaisundaram: School of Mathematics and Physics, University of Lincoln,
 	Brayford Pool, Lincoln LN6 7TS, United Kingdom}
 \email{anitha.t@cantab.net}

\thanks{The first author is supported by the Spanish Government grant MTM2017-86802-P, partly with FEDER funds, and by the “National Group for Algebraic and Geometric Structures and their
Applications” (GNSAGA - INdAM). She also acknowledges financial support from a London Mathematical Society Joint Research Groups in the UK (Scheme 3) grant. The second author acknowledges support from  EPSRC, grant EP/T005068/1.}

 \begin{abstract}
We show that the Hausdorff dimension of the closure of the second Grigorchuk group is $43/128$. Furthermore we establish that the second Grigorchuk group is super strongly fractal and that its automorphism group equals its normaliser in the full automorphism group of the tree.
 \end{abstract}
	
\maketitle

\section{Introduction}

Let $T$ be the $d$-adic tree,   for $d\ge 2$. The first Grigorchuk group \cite{GR}, more commonly known as the Grigorchuk group, was the earliest example of a regular branch group. The class of regular branch groups consists of subgroups of $\Aut(T)$ that mimic key properties of $\Aut(T)$; see Section~2 for precise definitions. The second Grigorchuk group was introduced in the same paper~\cite{GR}, and shares several interesting properties with the first Grigorchuk group, such as being infinite, periodic, finitely generated, just infinite, as well as possessing the congruence subgroup property, and having finitely many maximal subgroups only of finite index, see~\cite{Pervova,Pervova2}.

Regular branch groups  
have been studied extensively from various aspects over the past twenty years; see~\cite{Handbook} for a good introduction. The second Grigorchuk group however seems to have received less attention. We recall that the second Grigorchuk group acts on the $4$-adic tree, with generators $a$ and $b$, where $a$ cyclically permutes the four maximal subtrees rooted at the first level vertices, whereas $b$ fixes the first-level vertices pointwise and  is recursively defined by the tuple $(a,1,a,b)$ which corresponds to the action of~$b$ on the four maximal subtrees. The second Grigorchuk group was generalised by Vovkivsky~\cite{Vovkivsky} to the family of Grigorchuk-Gupta-Sidki (GGS-)groups acting on the $p^n$-adic tree, for $p$ any prime and $n\in\mathbb{N}$. Although the family of GGS-groups acting on the $p$-adic tree, for $p$ an odd prime, has been well studied (see for instance~\cites{Pervova3, Pervova4, FAZR2, FAGUA, Petschick, FT}), the more general GGS-groups, apart from in~\cite{Vovkivsky}, have only recently been considered in more depth (see~\cites{Elena-paper, DGT}).

In this note, we are interested in the Hausdorff dimension of regular branch groups acting on the  $4$-adic tree~$T$.  Let 
\[
\Gamma = \varprojlim_{n\in\mathbb{N}} C_{4} \wr \overset{n}\ldots \wr C_{4}
\]
be the group of all $4$-adic automorphisms, which is contained in a Sylow pro-$2$ subgroup of $\Aut(T)$. For $G\le \Gamma$, the Hausdorff dimension of the closure of $G$ in $\Gamma$ is given by
\[
\h_\Gamma (\overline{G})=\varliminf_{n\to \infty}\frac{\log| G:\St_G(n) |}{\log| \Gamma :\St_\Gamma(n)| }\in[0,1], 
\]
where $\varliminf$ represents the lower limit, and for a subgroup $H\le \Gamma$, the normal subgroup $\St_H(n)$ is the $n$th level stabiliser of~$H$ (also called the $n$th principal congruence subgroup of~$H$); we refer to Section 2 for further details. The Hausdorff dimension of $\overline{G}$ is a measure of how dense $\overline{G}$ is in $\Gamma$.  This was first applied by Abercrombie~\cite{Abercrombie}, Barnea and Shalev~\cite{BaSh97} in the more general setting of profinite groups.

Key results concerning the Hausdorff dimension of Grigorchuk-type groups were established in~\cites{AbVi05, Sunic, Siegenthaler, FAZR, FAZR2, FAGT}. It is proved in~\cite{NewHorizons} that the closure of the first Grigorchuk group has Hausdorff dimension $5/8$, however the computation of the Hausdorff dimension of the second Grigorchuk group does not appear to be recorded anywhere in the literature. Here, we close this gap.

\begin{theoremletter}\label{thm:main_intro}
Let $\mathcal{G}$ be the second Grigorchuk group acting on the $4$-adic tree~$T$. Then
\begin{enumerate}
    \item [(i)] the logarithmic orders of the congruence quotients of $\mathcal{G}$ are given by
    \[
    \log_4|\mathcal{G}:\St_{\mathcal{G}}(n)| =
    \begin{cases}
     1 & \mbox{ if } n=1, \\
     3 & \mbox{ if } n=2, \\
      \frac{1}{3}\big(  86\cdot4^{n-4} +4\big), & \mbox{ if } n>2;
    \end{cases}
    \]
    \item [(ii)] the Hausdorff dimension of the closure of $\mathcal{G}$ in $\Gamma$ is
    \[
    \h_\Gamma(\overline{\mathcal{G}})= 43/128.
    \]
\end{enumerate}
\end{theoremletter}

We also prove the following, which is of independent interest. We refer the reader to Section 2 for undefined terms.
\begin{propositionletter}\label{prop:main}
Let $\mathcal{G}$ be the second Grigorchuk group acting on the $4$-adic tree~$T$. Then the following hold:
\begin{enumerate}[label=(\roman*)]
    \item the group $\mathcal{G}$ is super strongly fractal;
    \item the group $\Aut(\G)$ equals the normaliser of~$\G$ in $\Aut(T)$.
\end{enumerate}
\end{propositionletter}

\medskip

\emph{Organisation.} Section~2 of this paper consists of background material for regular branch groups. Section~3 contains properties of the second Grigorchuk group~$\mathcal{G}$, and includes the proof of Proposition~\ref{prop:main}. Theorem~\ref{thm:main_intro} is proved in Section~4.

\subsection*{Acknowledgements}
The first author thanks the University of Lincoln for its excellent hospitality while this paper was being written. Both authors are grateful to the referee for very helpful comments.


\section{Preliminaries}

For $d\ge 2$, let $T$ be the $d$-adic tree,
  meaning that there is a distinguished vertex called the root and all vertices have $d$ children.  Using the
  alphabet $X = \{1,\ldots,d\}$, the vertices $u_\omega$ of $T$ are
  labelled bijectively by elements $\omega$ of the free
  monoid~$X^*$ as follows: the root of~$T$
  is labelled by the empty word~$\varnothing$, and for each word
  $\omega \in X^*$ and letter $x \in X$ there is an edge
  connecting $u_\omega$ to~$u_{\omega x}$.  We say
  that $u_\omega$ precedes $u_\lambda$ whenever $\omega$ is a prefix of $\lambda$.

  A natural length function on~$X^*$ is defined as follows: the words
  $\omega$ of length $\lvert \omega \rvert = n$, representing vertices
  $u_\omega$ that are at distance $n$ from the root, are the $n$th
  level vertices and form the \textit{$n$th layer} of the tree. The elements of the boundary $\partial T$ correspond
  naturally to infinite simple rooted paths.

  Denote by $T_u$ the full rooted subtree of $T$ that has its root at
  a vertex~$u$ and includes all vertices succeeding~$u$.  For any
  two vertices $u = u_\omega$ and $v = u_\lambda$, the map
  $u_{\omega \tau} \mapsto u_{\lambda \tau}$, induced by replacing the
  prefix $\omega$ by $\lambda$, yields an isomorphism between the
  subtrees $T_u$ and~$T_v$.  

  Every automorphism of $T$ fixes the root and the orbits of
  $\mathrm{Aut}(T)$ on the vertices of the tree $T$ are precisely its
  layers. For $f \in \mathrm{Aut}(T)$, the image of a vertex $u$ under
  $f$ is denoted by~$u^f$.  Observe that $f$ induces a faithful action
  on the monoid~$X^*$ such that
  $(u_\omega)^f = u_{\omega^f}$.  For $\omega \in X^*$ and
  $x \in X$ we have $(\omega x)^f = \omega^f x'$ where $x' \in X$ is
  uniquely determined by $\omega$ and~$f$.  This induces a permutation
  $f(\omega)$ of $X$ so that
  \[
  (\omega x)^f = \omega^f x^{f(\omega)}, \qquad \text{and hence}
  \quad   (u_{\omega x})^f = u_{\omega^f x^{f(\omega)}}.
  \]
  The automorphism $f$ is \textit{rooted} if $f(\omega) = 1$ for
  $\omega \ne \varnothing$.  It is \textit{directed}, with directed
  path $\ell \in \partial T$, if the support
  $\{ \omega \mid f(\omega) \ne 1 \}$ of its labelling is infinite and
  marks only vertices at distance $1$ from the set of
    vertices corresponding  to the path~$\ell$.
    
    When convenient, we do not differentiate between $X^*$ and vertices of~$T$. The \textit{section} of $f$ at a vertex~$u$ is the unique automorphism
$f_u$ of $T \cong T_{|u|}$ given by the condition $(uv)^f = u^f
v^{f_u}$ for $v \in X^*$.


\subsection{Subgroups of $\Aut(T)$}
Let $G\le \Aut(T)$. For $G\le \Aut(T)$, the
\textit{vertex stabiliser} $\St_G(u)$ is the subgroup
consisting of elements in $G$ that fix the vertex~$u$.
 For
$n \in \mathbb{N}$, the \textit{$n$th level stabiliser}
  $\St_G(n)= \bigcap_{\lvert \omega \rvert =n}
  \St_G(u_\omega)$
is the subgroup consisting of automorphisms that fix all vertices at
level~$n$.  Denoting by $T_{[n]}$ the finite subtree of $T$ on
vertices up to level~$n$, we see that $\St_G(n)$ is equal to
the kernel of the induced action of $G$ on $T_{[n]}$.

The full automorphism group $\mathrm{Aut}(T)$ is a profinite group:
\[
\mathrm{Aut}(T)= \varprojlim_{n\to\infty} \mathrm{Aut}(T_{[n]})
\]
The topology of $\mathrm{Aut}(T)$ is defined by the open subgroups
$\St_{\mathrm{Aut}(T)}(n)$, for $n \in \mathbb{N}$.  A
subgroup~$G$ of $\mathrm{Aut}(T)$ has the \textit{congruence subgroup
  property} if for every subgroup~$H$ of finite index in~$G$, there
exists some $n$ such that $\mathrm{St}_G(n)\subseteq H$.

Each $g\in \St_{\mathrm{Aut}(T)} (n)$ can be 
  described completely in terms of its restrictions to the subtrees
  rooted at vertices at level~$n$.  Indeed, there is a natural
  isomorphism
\[
\psi_n \colon \St_{\mathrm{Aut}(T)}(n) \rightarrow
\prod\nolimits_{\lvert \omega \rvert = n} \mathrm{Aut}(T_{u_\omega})
\cong \mathrm{Aut}(T) \times \overset{d^n}{\ldots} \times
\mathrm{Aut}(T).
\]
For ease of notation, we write $\psi=\psi_1$.

Let  $\omega\in X^*$ be of length $n$. We further define 
\[
\varphi_\omega :\mathrm{St}_{\mathrm{Aut}(T)}(u_\omega) \rightarrow \mathrm{Aut}(T_{u_\omega}) \cong \mathrm{Aut}(T)
\]
to be the natural restriction of $f\in \St_{\Aut(T)}(u_\omega)$ to the section~$f_{u_\omega}$.

A group $G \leq \Aut(T)$ is said to be \emph{self-similar} if the images under $\varphi_{\omega}$ and $\psi_n$ are contained in $G$ and $G \times \overset{d^n}{\dots} \times G$, respectively.

We say that the group $G$ is \textit{fractal} if $\varphi_\omega(\text{St}_G(u_\omega))=G$ for every $\omega\in X^*$. 
Furthermore we say that the group $G$ is \emph{strongly fractal} if $\varphi_x(\text{St}_G(1))=G$ for every $x\in X$, and we say that the group $G$ is \emph{super strongly fractal} if, for each $n\in \mathbb{N}$, we have $\varphi_\omega(\text{St}_G(n))=G$ for every word  $\omega\in X^*$ of length~$n$. For more information on these definitions and examples of groups satisfying these properties, see~\cite{Jone}.

To end this section, we recall that a level-transitive (i.e.\ transitive on every layer of the tree) self-similar group~$G$  with $1 \not = K \leq \St_G(1)$ such that $K\times \overset{d}\ldots \times K \subseteq \psi(K)$ and
$\lvert G : K \rvert < \infty$, 
is said to be \emph{regular
  branch over~$K$}.


\section{Some properties of the second Grigorchuk group}
Let $T$ be the $4$-adic tree. The second Grigorchuk group $\mathcal{G} \leq \Aut(T)$ is generated by two automorphisms $a$ and $b$, where $a$ is the rooted automorphism corresponding to the cycle $(1 \ 2 \ 3 \ 4)$, and $b \in \St_{\mathcal{G}}(1)$ is recursively defined by $\psi(b)=(a,1,a,b)$. The group $\G$ is periodic~\cite{GR}, and from~\cite[Proof of Lem.~2.1]{Pervova}, it follows that $\mathcal{G}$ is  regular branch over $\gamma_3(\G)$ and furthermore
\[
\psi(\gamma_3(\St_{\mathcal{G}}(1))) = \gamma_3(\G)\times \gamma_3(\G)\times\gamma_3(\G)\times \gamma_3(\G).
\]
Also we have $\St_{\G}(3) \leq \gamma_3(\G)$; see~\cite[Lem.~3.2]{Pervova}.  Note that $\St_{\G}(2)\not\le \gamma_3(\G)$, since 
\[
\psi(b(b^{a^2})^{-1})=(1,b^{-1},1,b)\in\St_{\G}(2),
\]
but $b(b^{a^2})^{-1}=[b^{-1},a^2]\not\in\gamma_3(\G)$ by~\cite[Lem.~2.1]{Pervova}.

\begin{lem}\label{lem:first-properties}
We have $|\G:\gamma_3(\G)|=4^3$ and $|\St_{\G}(1):\gamma_3(\St_{\G}(1))|=4^7$.
\end{lem}

\begin{proof}
The first statement follows from~\cite[Lem.~1.2 and~2.1]{Pervova}. For the second statement, note that 
\begin{align*}
\St_{\G}(1)/\St_{\G}(1)'&= \langle b,b^a,b^{a^2},b^{a^3}\rangle \St_{\G}(1)'\\
&\cong C_4\times C_4\times C_4\times C_4,
\end{align*}
and 
\begin{align*}
 \St_{\G}(1)'/ \gamma_3(\St_{\G}(1))&=  \langle [b,b^a], [b,b^{a}]^{a},[b,b^{a^3}]\rangle \gamma_3(\St_{\G}(1))\\
 &\cong C_4\times C_4\times C_4.
\end{align*}
Thus 
$$
|\St_{\G}(1):\gamma_3(\St_{\G}(1))|= |\St_{\G}(1):\St_{\G}(1)'|\cdot|\St_{\G}(1)':\gamma_3(\St_{\G}(1)')| = 4^7.
$$
\end{proof}

The above result has the following application.

\begin{lem}\label{lem:not-in-G}
For $\mathcal{G}$ the second Grigorchuk group, we have $\psi^{-1}(([b^{-1},a^2],1,1,1))\not\in\mathcal{G}$.
\end{lem}

\begin{proof}
For simplicity, we write $x=\psi^{-1}(([b^{-1},a^2],1,1,1))$. If $x$ were in $\mathcal{G}$, then clearly $x$  would be in $\St_{\G}(1)'$.

Recall that $[b^{-1},a^2]\not\in \gamma_3(\G)$. Hence $x\not\in \gamma_3(\St_{\G}(1))$. Therefore it suffices to show that the image of $x$ modulo 
$\gamma_3(\St_{\G}(1))$ does not lie in the finite abelian group of exponent~$4$ below:
\begin{align*}
 \St_{\G}(1)'/\gamma_3(\St_{\G}(1))
 &=  \langle [b,b^a], [b,b^{a}]^{a},[b,b^{a^3}]\rangle \gamma_3(\St_{\G}(1))\\
 &\cong \langle ([a,b],1,1,[b,a]), ([b,a],[a,b],1,1), (1,1,[a,b],[b,a])  \rangle\psi(\gamma_3(\St_{\G}(1))).
\end{align*}
From the above presentation, it is now clear that $x\not\in \St_{\G}(1)'$, and hence $x\not\in \G$, as required.
\end{proof}

\subsection{Partial weights}
In this subsection, we introduce  partial weights in order to give a complete description of the elements in $\St_{\G}(2)$ and $\St_{\G}(3)$, which we will need in the sequel. The definitions and the notation follow~\cite{FAZR2}.

\smallskip

We write $b_0=b$, $b_1=b^a$, $b_2=b^{a^2}$, and $b_3=b^{a^3}$. Recall that $\St_{\G}(1)=\langle b_0,b_1,b_2,b_3\rangle$. Hence for $g\in \St_{\G}(1)$, we can write $g$ as a word in $b_0,b_1,b_2,b_3$, that is,
\[
g=\omega(b_0,b_1,b_2,b_3),
\]
where $\omega=\omega(x_0,x_1,x_2,x_3)$  is a group word in four variables $x_0,x_1,x_2,x_3$.

\smallskip

For $\omega$ a group word in the variables $x_0,x_1,x_2,x_3$, 
\begin{enumerate}
    \item [(i)] the \emph{partial weight of $\omega$ with respect to $x_i$}, for $0\le i\le 3$, is the sum of the exponents of $x_i$ in $\omega$, reduced modulo~$4$, and
    \item [(ii)] the \emph{total weight of $\omega$} is the sum of all its partial weights, reduced modulo~$4$.
\end{enumerate}

\smallskip

Let $g=\omega(b_0,b_1,b_2,b_3)$ be an element of $\St_{\G}(1)$. For $0\le i\le 3$, write $r_i$ for the partial weight of $\omega$ with respect to $x_i$. Then
\begin{align}\label{eq:description}
\begin{split}
    \psi(g)=(a^{r_0+r_2}\omega_1(b_0,b_1,b_2,b_3), \, a^{r_1+r_3}\omega_2(b_0,b_1,b_2,b_3)\,,\qquad\qquad\qquad\qquad\qquad\qquad\\
 \qquad\qquad\qquad\qquad\qquad\qquad\, a^{r_0+r_2}\omega_3(b_0,b_1,b_2,b_3)\,,\, a^{r_1+r_3}\omega_0(b_0,b_1,b_2,b_3)) 
\end{split}
\end{align}
where  $\omega_i$ is a word of total weight $r_i$, for each $0\le i\le 3$.

\begin{thm}\label{thm:unique-rep}
Let $\mathcal{G}$ be the second Grigorchuk group, and let $g\in \St_{\G}(1)$. Then the partial and total weights are the same for all representations of $g$ as a word in $b_0,b_1,b_2,b_3$.
\end{thm}

\begin{proof}
We proceed as in~\cite[Proof of Thm.~2.8]{FAZR2}. First we note that it suffices to prove that, if $\omega$ is a word such that $\omega(b_0,b_1,b_2,b_3)=1$, then the total weight of $\omega$ is zero, and all partial weights $r_0,r_1,r_2,r_3$ of $\omega$ are also zero.

We observe from the expression for $\psi(g)$ above that 
\[
r_0+r_2=r_1+r_3=0,
\]
which proves that the total weight of $\omega$ is zero.

Additionally, since $\omega(b_0,b_1,b_2,b_3)=1$, in~\eqref{eq:description} we must have $\omega_i(b_0,b_1,b_2,b_3)=1$ for all $0\le i\le 3$. Since the total weight of such a $\omega_i$, which is $r_i$, has just been proved to be zero, the result follows.
\end{proof}

Let $g\in \St_{\G}(1)$. The \emph{partial weight of~$g$ with respect to $b_i$}, for $0\le i\le 3$, and the \emph{total weight of~$g$}, as the corresponding weights for any word representing~$g$. For $g$ with partial weights $r_i$ with respect to $b_i$ for $0\le i\le 3$, we further refer to $(r_0,r_1,r_2,r_3)\in (\mathbb{Z}/4\mathbb{Z})^4$ as the \emph{weight vector} of~$g$. 

The following is key.

\begin{thm}\label{thm:key}
Let $\mathcal{G}$ be the second Grigorchuk group and suppose $g\in \St_{\G}(1)$ has weight vector $(r_0,r_1,r_2,r_3)$. Then
\begin{enumerate}
    \item [(i)] we have $g\in\St_{\G}(2)$ if and only if $r_0+r_2 = r_1+r_3=0$;
    \item [(ii)] if $g\in \St_{\G}(3)$ then $r_0=r_1=r_2=r_3=0$.
\end{enumerate}
\end{thm}

\begin{proof}
(i) This is clear from (\ref{eq:description}).

(ii) Suppose that $g\in \St_G(3)$. From (\ref{eq:description}), it follows that $\omega_i (b_0,b_1,b_2,b_3)\in\St_{\G}(2)$ for all $0\le i\le 3$. Since the total weight for $\omega_i (b_0,b_1,b_2,b_3)$ is $r_i$, the result now follows from part~(i).
\end{proof}

\begin{thm}\label{thm:St-2}
Let $\mathcal{G}$ be the second Grigorchuk group. Then
\[
|\St_{\G}(2):\St_{\G}(3)|=2^{11}.
\]
\end{thm}

\begin{proof}
Let $g\in\St_{\G}(1)$ have weight vector $(r_0,r_1,r_2,r_3)$. From Theorem~\ref{thm:key}(i), it follows that there are exactly $4^2$ possibilities for $(r_0,r_1,r_2,r_3)$, for $g$ to be in $\St_{\G}(2)$. Using the notation of~\cite[Proof of Thm.~3.5]{FAZR2}, we denote the possibilities by
\[
r^{(i)}=(r_0^{(i)},r_1^{(i)},r_2^{(i)},r_3^{(i)})
\]
for each $1\le i\le 16$. For convenience, we write $r^{(1)}=(0,0,0,0)$.

We proceed as in~\cite[Proof of Thm.~3.5]{FAZR2}. First note that each solution $r^{(i)}$ determines a subset~$R^{(i)}$ of $\St_{\G}(2)$, consisting of all the elements whose weight vector is $r^{(i)}$. For the natural map $\pi:\mathcal{G}\rightarrow \mathcal{G}/\St_{\G}(3)$, we set $S^{(i)}=\pi(R^{(i)})$. From the previous paragraph, we have 
\[
\St_{\G}(2)/\St_{\G}(3)=\bigcup_{i=1}^{16} S^{(i)}.
\]
The theorem follows once we establish the following:
\begin{enumerate}
    \item [(a)] For $i\ne j\in \{1,\ldots,16\}$, the sets $S^{(i)}$ and $S^{(j)}$ are disjoint.
    \item [(b)] We have $|S^{(i)}|=2^7$ for all $i\in\{1,\ldots, 16\}$.
\end{enumerate}

To prove (a), we consider two elements $g\in R^{(i)}$ and $h\in R^{(j)}$ with the same image under~$\pi$. Then $gh^{-1}\in\St_{\G}(3)$ and so the weight vector of $gh^{-1}$ is the zero vector, by Theorem~\ref{thm:key}(ii). Since the weight vector defines a homomorphism from $\St_{\G}(1)$ to $(\mathbb{Z}/4\mathbb{Z})^4$, we deduce that $S^{(i)}=S^{(j)}$, and thus $i=j$, as required.

For part (b), we begin by observing that  $S^{(i)}$ is non-empty for each $1\le i \le 16$. Also, if $h_i$ is in $S^{(i)}$ then $S^{(i)}=h_iS^{(1)}$. Therefore $|S^{(i)}|=|S^{(1)}|$, and hence it suffices to prove that $|S^{(1)}|=2^7$.  

Let $g\in\St_{\G}(2)$. Then  $g\in R^{(1)}$ if and only if every component of $\psi(g)$ has zero total weight. Since $\G'$ consists of all the elements of $\St_{\G}(1)$ whose total weight is equal to $0$, it follows that $g\in R^{(1)}$ if and only if 
\[
\psi(g)\in \G'\times \G'\times \G'\times \G',
\]
and thus
\[
R^{(1)}=\G\cap \psi^{-1}(\G'\times \G'\times \G'\times \G').
\]
As in~\cite[Proof of Thm.~3.7]{FAZR2} (note that~\cite[Thm.~2.14]{FAZR2} holds with $p$ replaced by $4$), we deduce that
\[
|\G'\times \G'\times \G'\times \G':\psi(R^{(1)})|=4 \quad\text{ and }\quad R^{(1)}=\St_{\G}(1)'.
\]

Now, as in~\cite[Thm.~2.4(i)]{FAZR2}, we have $|\St_{\G}(1):\St_{\G}(2)|=4^2$. Thus
\[
|\G'\times \G'\times \G'\times \G':\St_{\G}(2)\times \St_{\G}(2)\times\St_{\G}(2)\times \St_{\G}(2) |=4^4,
\]
and it follows that $|S^{(1)}|=2^7$ once we establish that 
\[
|\St_{\G}(2)\times \St_{\G}(2)\times\St_{\G}(2)\times \St_{\G}(2):\psi(\St_{\G}(3)) |=2.
\]
As $\St_{\G}(3)=\St_{\G}(1)\cap \psi^{-1}(\St_{\G}(2)\times \St_{\G}(2)\times\St_{\G}(2)\times \St_{\G}(2))$, it suffices to show that
\[
|\psi^{-1}(\St_{\G}(2)\times \St_{\G}(2)\times\St_{\G}(2)\times \St_{\G}(2))\cdot \St_{\G}(1) : \St_{\G}(1)|=2.
\]
By Theorem~\ref{thm:key}, we have $\St_{\G}(2)=\langle b_0b_2^{-1}, b_1b_3^{-1}\rangle^G=\langle b_0b_2^{-1}\rangle^G=\langle [b^{-1},a^2]\rangle^G$. Since $[b^{-1},a^2]^2\in \gamma_{\G}(3)\le \St_{\G}(1)$, by Lemma~\ref{lem:not-in-G} it is enough to prove the following equality:
\[
\frac{\psi^{-1}(\St_{\G}(2)\times \St_{\G}(2)\times\St_{\G}(2)\times \St_{\G}(2))\cdot \St_{\G}(1) }{ \St_{\G}(1)}=\frac{\langle\psi^{-1}(([b^{-1},a^2],1,1,1))\rangle\St_{\G}(1)}{\St_{\G}(1)}
\]
Note that the expression above is clear from the fact that
\begin{align*}
\psi^{-1}((1,[b^{-1},a^2],1,1))&=\psi^{-1}(([b^{-1},a^2],1,1,1))\cdot [\psi^{-1}((b_0b_2^{-1},1,1,1)),a]\\
&=\psi^{-1}(([b^{-1},a^2],1,1,1))\cdot
\psi^{-1}((b_2b_0^{-1},b_0b_2^{-1},1,1))
\end{align*}
and 
\[
(b_2b_0^{-1},b_0b_2^{-1},1,1)=([a^2,b^{-1}],[b^{-1},a^2],1,1)\in \psi(\gamma_3(\G))\le \psi(\St_{\G}(1)).
\]
Hence we are done.
\end{proof}

\subsection{Further properties}\label{sec:furtherproperties}
We end this section with a few results of independent interest.

\begin{lem}
The second Grigorchuk group~$\mathcal{G}$ is  super strongly fractal.
\end{lem}

\begin{proof}
Observe that $\G$ is level-transitive. Also, the group $\G$ is fractal since $\psi(b)=(a,1,a, b)$ and $\psi(b^a) = (b,a,1,a)$. From ~\cite[Lem.~2.5]{Jone},  the group~$\G$ is also strongly fractal. 

We first show that $\varphi_{ux}(\St_{\G}(2))= \G$ for $u,x\in X$.
As seen in the proof of Theorem~\ref{thm:St-2}, we have $|\G:\St_{\G}(2)|=4^3$. Observe that $\St_{\G}(1)'\le \St_{\G}(2)$, hence it follows from Lemma~\ref{lem:first-properties}, and the remark above it, that 
\[
\frac{\St_{\G}(2)}{\St_{\G}(1)'}=\frac{\langle b(b^{a^2})^{-1}, b^a(b^{a^3})^{-1}\rangle\St_{\G}(1)'}{\St_{\G}(1)'}.
\]
We deduce that, for $u\in X$,
\[
\varphi_u(\St_{\G}(2))=\langle b\rangle G'.
\]
As $\psi(b)=(a,1,a,b)$ and $b^a=b[b,a]$, it follows that $\varphi_{ux}(\St_{\G}(2))=\G$, as required. 

Now let $n\in\mathbb{N}$ and write $K_n = \psi_n^{-1}(\gamma_3(\G) \times \overset{4^n}\ldots \times \gamma_3(\G))$. Since $\G$ is regular branch over $\gamma_3(\G)$, we have
$$
\psi_n(K_n) = \gamma_3(\G) \times \overset{4^n}{\dots} \times \gamma_3(\G) \subseteq (\St_{\G}(1) \times \overset{4^n}{\dots} \times \St_{\G}(1)) \cap \psi_n(\St_{\G}(n)) = \psi_n(\St_{\G}(n+1)).
$$
From the equation
\[
    \psi([b,a,a])=(b^{-1}ab^{-1}a,a^{-2}b,a^2,a^{-1}ba^{-1}),
 \]
 we deduce that $\varphi_x(\gamma_3(\G))\subseteq\langle a^2,b\rangle\G'\ne \G$, for all $x\in X$.  However from
 \[
 \psi([b,a,a^2])=(b^{-1},a^{-1}ba,b,a^{-1}ba)\in\psi(\St_{\G}(2)),
 \]
 we obtain that $\varphi_{ux}(\gamma_3(\G))=\G$ for all $u,x\in X$. 
 
 Write $N=\langle [b,a,a^2]\rangle^{\G}$, and $M_n=\psi_n^{-1}(N\times\overset{4^n}\ldots\times N)$. Then
 \[
 \psi_n(M_n)\subseteq (\St_{\G}(2) \times \overset{4^n}{\dots} \times \St_{\G}(2)) \cap \psi_n(\St_{\G}(n)) = \psi_n(\St_{\G}(n+2)).
 \]
 
 Hence for any $\omega\in X^*$ of length $n$, we have $N = \varphi_{\omega}(M_n) \subseteq \varphi_{\omega}(\St_{\G}(n+2))$. 
It remains to show that if we take $\omega\in X^*$ of length $n+2$, we have $\varphi_{\omega}(\St_{\G}(n+2)) = \G$, where $n \in\mathbb{N}$. 
But for any $u,x \in X$, and  any $\nu\in X^*$ of length $n$, we have $\G =\varphi_{ux}(N) \subseteq \varphi_{\nu ux}(\St_{\G}(n+2))$. From this it follows that for any $\omega\in X^*$ of length $n+2$ we have $\varphi_{\omega}(\St_{\G}(n+2))=\G$, as required.
\end{proof}

In what follows, we prove that the second Grigorchuk group is saturated. We recall that a group $G\le \Aut(T)$ is said to be \emph{saturated} if for any  $n\in\mathbb{N}$ there exists a subgroup $H_n \leq \St_G(n)$ that is characteristic in~$G$ and level-transitive on every $n$th level subtree. Examples of saturated groups acting on rooted trees are, among others, the first Grigorchuk group~\cite{LN}, the $p$-Basilica groups~\cite{DNT}, and the branch multi-EGS groups~\cite{TUA}.

\begin{thm}
The second Grigorchuk group $\G$ is saturated.
\end{thm}

\begin{proof}
We want to prove that,  for any $n\in\mathbb{N}$, there exists a subgroup $H_n \leq \St_{\G}(n)$ characteristic in~$\G$ and level-transitive on every subtree of the $n$th level. To this end, define the subgroups $H_n \leq \St_{\G}(n)$ inductively as follows:
$H_0=\G$, and $H_{n+1}=H_n'$. Then $H_1$ contains $\psi([a,b])=(b^{-1}a, a^{-1}, a, a^{-1}b)$ and $\psi([a,b]^{a^2})=(a,a^{-1}b,b^{-1}a,a^{-1})$.  We deduce that $\varphi_1(H_1)=\G$, and thus $\varphi_x(H_1)=\G$ for any $x\in X$. Hence $H_1$ acts level-transitively on all subtrees rooted at a level one vertex. By induction it follows that the restriction of $H_n$ to subtrees rooted at an $n$th level vertex contains~$\G$ and thus the required level-transitivity of~$H_n$ follows.
\end{proof}

As a straightforward application of~\cite[Thm.~7.5]{LN}, since $\G$ is saturated we obtain the following. 
\begin{cor}
The automorphism group $\Aut(\G)$ of the second Grigorchuk group coincides with the normaliser of $\G$ in $\Aut(T)$, that is, $\Aut(\G)=N_{\Aut(T)}(\G)$.
\end{cor}


\section{Hausdorff dimension of the closure of $\mathcal{G}$}

This section is devoted to determining the Hausdorff dimension of the closure of~$\mathcal{G}$.

\begin{proof}[Proof of Theorem~\ref{thm:main_intro}]
(i) For convenience, we write $\G_n= \G/\St_{\G}(n)$. It is easy to see that $|\G_1|=4$. For $n=2$, as stated in the proof of Theorem~\ref{thm:St-2}, we have $|\G_2|=4^3$.
For $n=3$, the result follows from Theorem~\ref{thm:St-2}.

Now we let $n \geq 4$. As $\St_{\G}(3) \leq \gamma_3(\G)$ and $\mathcal{G}$ is regular branch over $\gamma_3(\G)$, we have
\begin{align*}
|\G_n|
& =|\G:\St_{\G}(1)|\cdot|\St_{\G}(1):\St_{\G}(n)| \\
&= 4\cdot \frac{|\G \times \overset{4}{\dots} \times \G :\St_{\G}(n-1) \times \overset{4}{\dots} \times \St_{\G}(n-1)|}{|\G \times \overset{4}{\dots} \times \G : \psi(\St_{\G}(1))}| \\
&= 4 \cdot \frac{|\G_{n-1}|^4}{|\G \times \overset{4}{\dots} \times \G : \psi(\St_{\G}(1))}|.
\end{align*}

Since
\begin{align*}
|\G \times \overset{4}{\dots} \times \G : \psi(\St_{\G}(1))| &= \frac{|\G \times \overset{4}{\dots} \times \G : \gamma_3(\G) \times \overset{4}{\dots} \times \gamma_3(\G)|}{|\psi(\St_{\G}(1)): \gamma_3(\G) \times \overset{4}{\dots} \times \gamma_3(\G)|}\\
&= \frac{|\G : \gamma_3(\G)|^4}{|\St_{\G}(1) :\gamma_3(\St_{\G}(1))|}\\
&= 4^{5},
\end{align*}
we have
\[
\log_4|\G_n|=-4+4\log_4|\G_{n-1}|,
\]
and the result follows by induction.

(ii) This follows immediately from (i). Indeed, if $n>2$, we have
\begin{align*}
\text{hdim}_\Gamma(\overline{\mathcal{G}})&=\varliminf_{n\to \infty}\,\frac{\log_4|\G_n|}{\log_4 |\Gamma:\St_{\Gamma}(n)|}\\
    &=\varliminf_{n\to \infty}\,\frac{\frac{1}{3}\big(  86\cdot4^{n-4} +4\big)}{\frac{4^n}{3}}\\
    &= \frac{86}{4^4}=\frac{43}{128},
\end{align*}
as required.
\end{proof}

\end{document}